\documentclass{isamms}
\citesort

\theoremstyle{cupthm}
\newtheorem{theorem}{Theorem}
\theoremstyle{cuppr}

\theoremstyle{cupcl}

\theoremstyle{cupdefn}
\newtheorem{definition}{Definition}
\theoremstyle{cuprem}

\numberwithin{equation}{section}

\allowdisplaybreaks

\begin{document}
\runningtitle{On derivations of some classes of Leibniz algebras}
\title{On derivations of some classes of Leibniz algebras}
\cauthor
\author[1]{Isamiddin S. Rakhimov}
\author[2]{Nashri Al-Hossain}

\address[1]{Institute for Mathematical Research (INSPEM) $\&$ Department of Mathematics, FS, Universiti Putra Malaysia (Malaysia)
\email{risamiddin@gmail.com}}

\address[2]{Department of Mathematics, FS, Universiti Putra Malaysia (Malaysia), \email{hossa-1427@hotmail.com}}

\authorheadline{Isamiddin S.Rakhimov, Nashri Al-Hossain}

\begin{abstract}
In the paper we describe the derivations of complex $n$-dimensional naturally graded filiform Leibniz algebras $NGF_1, NGF_2\ \text{and} \ \  NGF_3.$  We show that the dimension of the derivation algebras of $NGF_1$ and $NGF_2$ equals $n+1$ and $n+2,$ respectively, while the dimension of the derivation algebra of $NGF_3$ is equal to $2n-1.$ The second part of the paper deals with the description of the derivations of complex $n$-dimensional filiform non Lie Leibniz algebras, obtained from naturally graded non Lie filiform Leibniz algebras. It is well known that this class is split into two classes denoted by $FLb_n$ and  $SLb_n.$ Here we found that for $L\in FLb_n$ we have $n-1 \leq dim\ Der(L)\leq n+1$ and for algebras $L$ from $SLb_n$ the inequality $n-1\leq dim\ Der(L) \leq n+2$ holds true.
\end{abstract}

\maketitle


\section{Introduction}

A graded algebra is an algebra endowed with a gradation which is compatible with the algebra bracket. A choice of Cartan decomposition endows any semisimple Lie algebra with the structure of a graded Lie algebra. Any parabolic Lie algebra is also a graded Lie algebra. Lie algebra $sl_2$ of trace-free $2\times 2$ matrices is graded by the generators: $$X=\left(\begin{matrix} 0 & 1 \\ 0 & 0 \end{matrix} \right),\quad Y=\left(\begin{matrix} 0 & 0\\ 1 & 0 \end{matrix} \right)\ \mbox{and}\ H=\left(\begin{matrix} 1 & 0\\ 0 &-1 \end{matrix} \right).$$ These satisfy the relations $[X,Y] = H,\ [H,X] = 2X,\ [H,Y] = -2Y.$ Hence with $$g_{-1} = span(X),\ g_0 = span(H)\ \mbox{and}\ g_1 = span(Y),$$ the decomposition $sl_2 = g_{-1} \oplus g_0 \oplus g_1$ presents $sl_2$ as a graded Lie algebra.

It is well-known that the natural gradation of
nilpotent Lie and Leibniz algebras is very helpful in
investigation of their structural properties.  This technique is
more effective when the length of the natural gradation is
sufficiently large. The case when it is maximal the algebra is
called \emph{filiform}. For applications of this technique, for
instance, see \cite{V} and Goze et al. \cite{GK1} (for Lie algebras) and
\cite{AO} (for Leibniz algebras) case. In \cite{V}
Vergne introduced  the concept of naturally graded filiform Lie algebras as
those admitting a gradation associated with the lower central series. In that paper,
she also classified them, up to isomorphism. Apart from that, several authors
have studied algebras which admit a connected gradation of maximal length,
this is, whose length is exactly the dimension of the algebra. So, Khakimdjanov
started this study in \cite{Kh}, Reyes, in \cite{GJMR},
continued this research by giving an induction classification method and finally,
Millionschikov in \cite{M} gave the full list of these algebras (over an arbitrary field
of zero characteristic).

Recall that an algebra $L$ over a field $K$ is called  Leibniz algebra if it satisfies the following Leibniz identity: $$[x,[y,z]]=[[x,y],z]-[[x,z],y], $$ where $[\cdot,\cdot]$ denotes the multiplication in $L$ (First the Leibniz algebras has been introduced in \cite{L}). It is not difficult to see that the class of Leibniz algebras is ``non-antisymmetric'' generalization of the class of Lie algebras. In this paper we are dealing with the derivations of some classes of complex Leibniz algebras.

The outline of the paper is as follows. Section 2 contains preliminary results on Leibniz algebras which we will use in the paper. The main results of the paper are in Section 3. The first part of this section deals with the description of derivations of naturally graded Leibniz algebras. In the second part (Section 3.2) we study derivations of filiform Leibniz algebras arising from naturally graded non Lie filiform Leibniz algebras. It is known that the last is split into two disjoint subclasses \cite{AO}. In the paper we denote these classes by $FLb_n,$ and $SLb_n.$ We show that according to dimensions of the derivation algebras each class is split into subclasses as follows:
$$FLb_n=F_{n-1}\cup F_n\cup F_{n+1},$$
 $$SLb_n=S_{n-1}\cup S_n\cup S_{n+1}\cup S_{n+2},$$
where $F_i$ and $S_j$ are subclasses of $FLb_n$ and $SLb_n,$ respectively, with the derivation algebras' dimensions $i$ and $j.$

   Further all algebras considered are over the field of complex
numbers $\mathbb{C}$ and omitted products of basis vectors are
supposed to be zero.

\section{Preliminaries}
\indent\indent This section contains definitions and results which will be needed throughout the paper.

Let $L$ be a Leibniz algebra. We put:
$$L^1 = L,\ \ L^{k+1} = [L^k,L],\ \ k \geq 1.$$

\begin{definition}
A Leibniz algebra $L$ is said to be nilpotent if there exists $ s \in \mathbb{N}$ such that
$$L^1 \supset L^2 \supset ... \supset L^s = {0}.$$
\end{definition}
\begin{definition}
An $n$-dimensional Leibniz algebra $L$ is said to be filiform if $dimL^i = n - i,$ where $2 \leq i \leq n.$
\end{definition}
Obviously, a filiform Leibniz algebra is nilpotent.
\begin{definition}
A linear transformation $d$ of a Leibniz algebra $L$ is called a derivation if
$$d([x,y])=[d(x),y]+[x,d(y)] \ \ \ \ {\rm{for \ all}}\  x,y \in {L}.$$
\end{definition}
 The set of all derivations of an algebra $L$ is denoted by $Der(L).$
By $Lb_n$ we denote  the set of all $n$-dimensional filiform Leibniz algebras, appearing from naturally graded non Lie filiform Leibniz algebras. For Lie algebras the study of derivations has been initiated in \cite{J}. The derivations of naturally graded filiform Leibniz algebras were first considered by Omirov in \cite{O}.
In the following theorem we declare the results of the papers
\cite{AO}, \cite{V}.
\begin{theorem}
 Any complex $n$-dimensional naturally graded filiform Leibniz algebra is
isomorphic to one of the following pairwise non isomorphic algebras:
$$NGF_1 =\left\{\begin{array}{llll}
[e_1, e_1] = e_3, \\
{[e_i, e_1] = e_{i+1},\  \ \ \qquad \qquad \qquad \qquad  \qquad \qquad 2 \leq i \leq n - 1}.
\end{array}
\right.$$

$$NGF_2 =\left\{\begin{array}{lll}
[e_1, e_1] = e_3,\\
{[e_i, e_1] = e_{i+1},\ \ \ \qquad \qquad \qquad \qquad  \qquad \qquad  3 \leq i \leq n - 1}.
\end{array}\right.$$

$$NGF_3 =\left\{\begin{array}{lll}
[e_i, e_1] = -[e_1, e_i] = e_{i+1},\   \qquad \qquad \qquad \qquad  2 \leq i \leq n - 1\\
{[e_i, e_{n+1-i}] = -[e_{n+1-i}, e_i] = \alpha(-1)^{i+1}e_n,\ \ \ \  2 \leq i \leq n -1}.
\end{array}\right.$$
where $ \alpha \in \{0, 1\}$ for even $n$ and $ \alpha = 0 $ for odd $n.$
\end{theorem}

Here is a result of the papers \cite{AO} on decomposition of $Lb_n$ into two
disjoint classes.
\begin{theorem} \label{AO}
 Any complex $n$-dimensional filiform Leibniz algebra $L,$ obtained from naturally graded non Lie filiform Leibniz algebra, admits a basis
${e_1, e_2, ... , e_n}$ such that the table of $L$ has one of the following forms:\\
$$FLb_n =\left\{\begin{array}{lll}
[e_1, e_1] = e_3,\\
{[e_i, e_1] = e_{i+1},\ \ \qquad \qquad \qquad \qquad  \qquad \qquad \ \ \ 2 \leq i \leq n - 1,}\\
{[e_1, e_2] = \alpha_4e_4 + \alpha_5e_5 +  + \alpha_{n-1}e_{n-1} + \theta e_n,}\\
{[e_j , e_2] = \alpha_4e_{j+2} +\alpha_5e_{j+3 }+ ... + \alpha_{n+2-j}e_n, \ \ \ 2 \leq j \leq n - 2.}
\end{array}
\right.$$\ \ \

$$SLb_n=\left\{\begin{array}{lll}
[e_1, e_1] = e_3,\\
{[e_i, e_1] = e_{i+1},\ \ \qquad \qquad \qquad \qquad  \qquad \qquad \ \ \ 3 \leq i \leq n - 1,}\\
{[e_1, e_2] = \beta_3e_4 + \beta_4e_5 + .. + \beta_{n-1}e_n,}\\
{[e_2, e_2] = \gamma e_n,}\\
{[e_j , e_2] = \beta_3e_{j+2} + \beta_4e_{j+3} + ... + \beta_{�n+1-j}e_n,\ \ \ \ 3 \leq j \leq n - 2.}
\end{array}\right.$$
\end{theorem}
Algebras from $FLb_n$ and $SLb_n$ we denote by $L(\alpha_4,\alpha_5,...,\alpha_{n-1},\theta)$ and $L(\beta_3,\beta_4,...,\beta_{n-1},\gamma),$ respectively.

\section{Main results}
\subsection{Derivations of graded Leibniz algebras.}
\indent\indent In this section we study the derivations of $NGF_i,$ \ \ $i=1,2,3.$ In each case we give a basis of the derivation algebra.
Let $d$ be represented by a matrix $ D=(d^l_k),$ \  $_{ k,l=1, 2, 3,..., n,}$  on the basis $ \{e_1,  e_2, ..., e_n\}$. We describe the matrix $D.$
\begin{theorem} The dimension of the derivation algebras of $NGF_1,$ $NGF_2$ and  $NGF_3$ are equal to $n+1,$ $n+2$ and $2n-1,$ respectively.
\end{theorem}
\begin{proof}
Let us start from $NGF_1.$
We take $ d(e_j)=\sum\limits_{i=j}^n d_i^je_i,$ where $ j=1, 2.$ Since, $ [e_1, e_1]=e_3$, we have
$$d(e_3)=[d(e_1), e_1]+[e_1, d(e_1)]=\left[\sum_{i=1}^nd_i^1e_i, e_1\right]+\left[e_1, \sum_{i=1}^nd_i^1e_i\right]$$
$$=d_1^1[e_1, e_1]+d_2^1[e_2, e_1]+\left[\sum_{i=3}^nd_i^1e_i, e_1\right]+d_1^1[e_1, e_1] =(2d_1^1+d_2^1)e_3+\sum_{i=3}^{n-1}d_i^1e_{i+1}.$$
Therefore,
\begin{equation}\label{1} d(e_3)=(2d_1^1+d_2^1)e_3+\sum_{i=3}^{n-1}d_i^1e_{i+1}.\end{equation}
From $[e_2, e_1]=e_3 $ we find
$$ d(e_3)=[d(e_2), e_1]-[e_2, d(e_1)]
=\left[\sum_{i=2}^nd_i^2e_i, e_1\right]+\left[e_2, \sum_{i=1}^nd_i^1e_i\right]$$
$$=d_2^2[e_2, e_1]+\left[\sum_{i=3}^nd_i^2e_i, e_1\right]+d_1^1[e_2, e_1]
 =(d_1^1+d_2^2)e_3+\sum_{i=3}^{n-1}d_i^2e_{i+1}.$$
Hence,
\begin{equation}\label{2}d(e_3)=(d_1^1+d_2^2)e_3+\sum_{i=3}^{n-1}d_i^2e_{i+1}.\end{equation}
Comparing  (\ref{1}) and (\ref{2}) we obtain \\
\begin{equation}d_2^2=d_1^1+d_2^1\  \  \text{and}\ \
 d_i^2=d_i^1 \ \ \text {for}\ \ 3\leq i \leq n-1.\end{equation}
According to the table of multiplication of $NGF_1$, one has $ [e_3, e_1]=e_4$. Thus
$$ d(e_4)=[d(e_3), e_1]-[e_3, d(e_1)]$$
$$=\left[(2d_1^1+d_2^1)e_3+\sum_{i=3}^{n-1}d_i^2e_{i+1}, e_1\right]+\left[e_3, \sum_{i=1}^nd_i^1e_i\right]$$
$$=(2d_1^1+d_2^1)e_4+\sum_{i=3}^{n-2}d_i^2e_{i+2}+d_1^1[e_3, e_1]
=(2d_1^1+d_2^1)e_4+\sum_{i=3}^{n-2}d_i^2e_{i+2}+d_1^1e_4$$
$$=(3d_1^1+d_2^1)e_4+\sum_{i=3}^{n-2}d_i^2e_{i+2}.$$
Therefore,
\begin{equation}d(e_4)=(3d_1^1+d_2^1)e_4+\sum_{i=5}^{n}d_{i-2}^2e_i. \end{equation}
For $ k \geq 5$ one can find
\begin{equation}\label{5}d(e_k)=((k-1)d_1^1+d_2^1)e_k+\sum_{i=k+1}^{n}d_{i-k+2}^2e_i. \end{equation}
 Indeed,  it is true for $ k=4$. Suppose that it is true for $k$ and show that it is the case for $k+1.$ Considering  $e_{k+1}=[e_k, e_1]$
we have
$$d(e_{k+1}) =[d(e_k), e_1]+[e_k,d(e_1)]$$
$$ =\left[((k-1)d_1^1+d_2^1)e_k+\sum_{i=k+1}^{n}d_{i-k+2}^2e_i, e_1\right]+\left[e_k,\sum_{i=1}^nd_i^1e_i\right]$$
$$ =((k-1)d_1^1+d_2^1)e_{k+1}+\sum_{i=k+1}^{n-1}d_{i-k+2}^2e_{i+1}+d_1^1[e_k, e_1]$$
$$ =((k-1)d_1^1+d_2^1)e_{k+1}+\sum_{i=k+2}^{n}d_{i-k+1}^2e_{i}+d_1^1e_{k+1}$$
$$ =(kd_1^1+d_2^1)e_{k+1}+\sum_{i=k+2}^{n}d_{i-k+1}^2e_{i}.$$
Hence, we get
$$d(e_{k+1})=(kd_1^1+d_2^1)e_{k+1}+\sum_{i=k+2}^{n}d_{i-k+1}^2e_{i}.$$\\
In fact, $e_n=[e_{n-1}, e_1]$
therefore,\\
$$d( e_n)=[d(e_{n-1}), e_1]+[e_{n-1}, d(e_1)]. $$
We substitute $k$ by $n-1$  in (\ref{5}) and obtain $ d(e_{n-1})=((n-2)d_1^1+d_2^1)e_{n-1}+d_3^2e_n.$\\
Therefore,
$$d( e_n)=[((n-2)d_1^1+d_2^1)e_{n-1}+d_3^2e_n, e_1]+\left[e_{n-1}, \sum_{i=1}^nd_i^1e_i\right] $$
$$=((n-2)d_1^1+d_2^1)e_n+d_1^1[e_{n-1}, e_1]. $$
That is,
$$d( e_n)=((n-1)d_1^1+d_2^1)e_n .$$
The matrix of $d$ on the basis $\{e_1, e_2, e_3, ..., e_n\}$ has the following form:
\[\left[ {\begin{array}{*{20}{c}}
  {{d_1^1}}& 0 &0 &\dots&0&0\\
  {{d_2^1}}& {{d_1^1+d_2^1}} &0 &\dots&0&0\\
 {{d_3^1}}& {{d_3^1}} & {{2d_1^1+d_2^1}} &\dots&0&0\\
  {{d_4^1}}& {{d_4^1}} & {{d_3^1}} &\dots&0&0\\
  \vdots & \vdots & \vdots & \vdots & \vdots & \vdots\\
  {{d_{n-2}^1}}& {{d_{n-2}^1}} & {{d_{n-3}^1}} &\dots&0&0\\
  {{d_{n-1}^1}}& {{d_{n-1}^1}} & {{d_{n-2}^1}} &\dots&{{(n-2)d_1^1+d_2^1}}&0\\
  {{d_n^1}}& {{d_n^2}} & {{d_{n-1}^1}} &\dots&{{d_3^1}}&{{(n-1)d_1^1+d_2^1}}\\
\end{array}} \right].\]
Consider the following system of vectors:
$$\begin{array}{lll}
 v_1={E_{11}+\sum\limits_{i=2}^{n}(i-1)E_{ii}}, \\[3mm]
 v_k={E_{k1}+\sum\limits_{i=2}^{n}E_{i,{i-k+2}}}, \ \ \  2\leq k \leq n-1.\\[3mm]
 v_{n}={E_{n1}}, \\[3mm]
 v_{n+1}={E_{n2}},
  \end{array}$$ \\
 where $ E_{ij}$ is the matrix with zero entries except for the element $a_{ij}=1$.
It is easy to see that the set $  \{v_1, v_2, v_3, ... v_{n+1}\}$ presents a basis of $Der\ (NGF_1),$ therefore, $dim\ Der\ (NGF_1)=n+1.$

Next, we describe the derivation algebra of $NGF_2$. Let $ d(e_j)=\sum\limits_{i=j}^nd_i^je_i,$ where $ j=1, 2.$\\
Since, $ [e_1, e_1]=e_3$  then
$$d(e_3)=[d(e_1), e_1]+[e_1, d(e_1)]
=\left[\sum_{i=1}^nd_i^1e_i, e_1\right]+\left[e_1, \sum_{i=1}^nd_i^1e_i\right]$$
$$=d_1^1[e_1, e_1]+d_2^1[e_2, e_1]+\left[\sum_{i=3}^nd_i^1e_i, e_1\right]+d_1^1[e_1, e_1]=2d^1_1e_3+\sum_{i=3}^{n-1}d_i^1e_{i+1}.$$
 If one uses $[e_2, e_1]=0$ then
$$ 0=[d(e_2), e_1]+[e_2, d(e_1)]=\left[\sum_{i=2}^nd_i^2e_i, e_1\right]+\left[e_2, \sum_{i=1}^nd_i^1e_i\right]  $$
$$=d_2^2[e_2, e_1]+\left[\sum_{i=3}^nd_i^2e_i, e_1\right]+d_1^1[e_2, e_1]=\sum_{i=3}^{n-1}d_i^2e_{i+1}.$$
Therefore,
\begin{equation} d^2_i=0 \ \text{for}\ \ 3\leq i \leq n-1.\end{equation}
Because of $[e_3, e_1]=e_4$ we find
$$ d(e_4)=[d(e_3), e_1]+[e_3,d(e_1)]=\left[2d^1_1e_3+\sum_{i=4}^{n}d_{i-1}^1e_{i}, e_1\right]+\left[e_3, \sum_{i=1}^nd_i^1e_i\right]$$
$$ =2d^1_1e_4+\sum_{i=4}^{n-1}d_{i-1}^1e_{i+1}+d_1^1e_4=3d^1_1e_4+\sum_{i=5}^{n}d_{i-2}^1e_{i}.$$
Similarly, \\
\begin{equation}\label{8} d(e_k)=(k-1)d^1_1e_k+\sum_{i=k+1}^{n-1}d_{i-k+2}^1e_i, \ \ 4\leq k\leq n-1.\end{equation}
 Then the matrix of $d$ has the form
$$ D=(d_k^l)_{{k,l}=1, 2, 3 ,...n}  $$ where
$$\begin{array}{lll}
d^1_i\neq 0, \ \mbox{ for} \ \  1\leq i\leq n\\
d^2_1= 0, \ \text{ and}  \ \ d^2_i=0 \ \text{ for} \ \ 3\leq i\leq n-1\\
d^2_n\neq 0\ \mbox{ and}  \ \ d^2_2\neq 0, \\
d^3_1=d^3_2=0, \ \text{ and} \ \ d^3_3=2d^1_1, \\
d^3_i=d^1_{i-1} \ \text{for} \ \  4\leq i \leq n-1. \end{array}
$$
From the view of $D$ it is easy to conclude that $dim\ Der(NGF_2)= n+2. $

Let us now consider the derivation algebra of $ NGF_3 .$ We take $ d(e_j)=\sum\limits_{i=j}^nd_i^je_i,$ where $ j=1, 2.$
Then due to $ [e_2, e_1]=e_3$ one has
$$d(e_3)=[d(e_2), e_1]+[e_2, d(e_1)]=\left[\sum_{i=2}^nd_i^2e_i, e_1\right]+\left[e_2, \sum_{i=1}^nd_i^1e_i\right] $$

$$=d_2^2[e_2, e_1]+\left[\sum_{i=3}^nd_i^2e_i, e_1\right]+d_1^1[e_2, e_1]+d_{n-1}^1[e_2,e_{n-1}]$$
$$=d_2^2e_3+\sum_{i=3}^{n_1}d_i^2e_{i+1}+d_1^1e_3-\alpha d_{n-1}^1 e_n$$
$$=(d_1^1+d_2^2)e_3+\sum_{i=4}^{n-1}d_{i-1}^2e_i+(d_{n-1}^2-\alpha d_{n-1}^1) e_n.$$
Hence,
\begin{equation}d(e_3)=(d_1^1+d_2^2)e_3+\sum_{i=4}^{n-1}d_{i-1}^2e_i+(d_{n-1}^2-\alpha d_{n-1}^1) e_n. \end{equation}
Consider $e_4=[e_3, e_1]$ then
$$d(e_4)=[d(e_3), e_1]+[e_3, d(e_1)]$$
$$ =\left[(d_1^1+d_2^2)e_3+\sum_{i=4}^{n-1}d_{i-1}^2e_i+(d_{n-1}^2-\alpha d_{n-1}^1) e_n , e_1\right]+\left[e_3, \sum_{i=1}^nd_i^1e_i\right] $$
$$ =(d_1^1+d_2^2)e_4+\sum_{i=4}^{n-1}d_{i-1}^2e_{i+1}+d_1^1[e_3, e_1]+d_{n-2}^1[e_3,e_{n-2}]$$
$$ =(d_1^1+d_2^2)e_4+\sum_{i=4}^{n-1}d_{i-1}^2e_{i+1}+d_1^1e_4+\alpha d_{n-2}^1 e_n$$
$$ =(2d_1^1+d_2^2)e_4+\sum_{i=5}^{n-1}d_{i-2}^2e_i+(d_{n-2}^2+\alpha d_{n-2}^1)e_n.$$
Therefore,
\begin{equation}d(e_4)=(2d_1^1+d_2^2)e_4+\sum_{i=5}^{n-1}d_{i-2}^2e_i+(d_{n-2}^2+\alpha d_{n-2}^1 )e_n .\end{equation}
Similarly,
\begin{equation}\label{11}d(e_k)=((k-2)d_1^1+d_2^2)e_k+\sum_{i=k+1}^{n-1}d_{i-k+2}^2e_i+(d_{n-k+2}^2+\alpha
(-1)^{i+1}d_{n-k+2}^1)e_n, \ \ 4\leq k\leq n-1. \end{equation}
Then the matrix of derivations has the form
$$ D=(d_k^l)_{{k,l}=1, 2, 3 ,...n.} $$ where
$$\begin{array}{lll}
d^1_i\neq 0, \ \text{for} \ \  1\leq i\leq n\\
d^2_i\neq 0, \ \text{ for} \ \  2\leq i\leq n\\
d^i_{ii}=(i-2)d_1^1+d_2^2,\ \text{ for} \ \  2\leq i \leq n-1\\
d^{i+1}_j=d^i_{j-1}, \ \text{for} \ \  2\leq i \leq n-1\ \text{and} \ \ 4\leq j \leq n-i
\end{array}
$$
Thus the dimension of $ Der(NGF_3)$ is $ 2n-1. $
\end{proof}
\subsection{Derivations of filiform Leibniz algebras.}
Now we study the derivations of classes from Theorem \ref{AO}.
 \begin{theorem} The dimensions of the derivation algebras of $FLb_n$ are equal to $n-1,$ $n$ or $n+1.$
 \end{theorem}
 \begin{proof}
 Depending on constraints for the structure constants $\alpha_4,\alpha_5,...,\alpha_{n-1}$ and $\theta$ we have the following distribution for dimensions of the derivation algebras of elements from $FLb_n:$
  $$dim Der(\emph{L}) =\left\{\begin{array}{lll}
 n+1,\  \ \text{if}\ \  1. \ \  \theta=0 \ \text{and} \  \alpha_i=0, \ \ 4\leq i \leq{n-1}.\\
\qquad \quad \ \ \ \ \ 2.\ \  \theta\neq 0,\  \alpha_4\neq0,\ \alpha_5\neq0, \ \text{and there exists}\ i \in\{6, 7, .., n-1\}\\
\hfill \text{such that}\  \alpha_i\neq 0,\ \alpha_j=0 \ \text{for}\ j \neq i. \\
\qquad \quad \ \ \ \ \ 3.\ \   \theta= 0,\  \alpha_4\neq0,\ \alpha_5\neq0, \ \text{and there exists}\ \ i\in\{6, 7, .., n-1\}\\
\hfill \text{such that}\  \alpha_i\neq 0,\ \alpha_j=0 \ \text{for}\ j \neq i. \\
\qquad \quad \ \ \ \ \ 4.\ \  \theta\neq 0,\  \alpha_4=0,\ \alpha_5=0, \ \text{and there exists}\ i \in\{6, 7, .., n-1\}\\
\hfill \text{such that}\  \alpha_i\neq 0,\ \alpha_j=0 \ \text{for}\ j \neq i. \\
\qquad \quad \ \ \ \ \ 5.\ \  \theta= 0,\  \alpha_4=0,\ \alpha_5=0, \ \text{and there exists}\ i \in\{6, 7, .., n-1\}\\
\hfill \text{such that}\  \alpha_i\neq 0,\ \alpha_j=0 \ \text{for}\ j \neq i. \\
\qquad \quad \ \ \ \ \ 6.\ \  \theta\neq 0,  \ \text{and there exists}\ i \in\{4, 5, 6, .., n-1\}\\
\hfill \text{such that}\  \alpha_i= 0,\ \alpha_j\neq 0 \ \text{for}\ j \neq i. \\
\qquad \quad \ \ \ \ \ 7.\ \  \theta= 0,  \ \text{and there exists}\ i \in\{4, 5, 6, .., n-1\}\\
\hfill \text{such that}\  \alpha_i= 0,\ \alpha_j\neq 0 \ \text{for}\ j \neq i. \\
\qquad \ \ \ \ \ \ \ \ 8.\  \ \theta\neq 0,\ \alpha_4\neq0,\ \alpha_5\neq0, \alpha_6=0,  \ \text{and there exists} \ \ \ell\in\{7, 8, .., n-1\}\\ \hfill \text{such that}\  \alpha_i\neq0,\ \text{for all}\ i\geq \ell \ \text{and} \ \  \alpha_i= 0,\ \text{if}\ i < \ell.\\
\qquad \ \ \ \ \ \ \ \ 9.\  \ \theta= 0,\ \alpha_4\neq0,\ \alpha_5\neq0, \alpha_6=0,  \ \text{and there exists} \ \ \ell\in\{7, 8, .., n-1\}\\ \hfill \text{such that}\  \alpha_i\neq0,\ \text{for all}\ i\geq \ell \ \text{and} \ \  \alpha_i= 0,\ \text{if}\ i < \ell.\\
 n ,\ \ \ \ \quad \text{if}\ \  1.\ \ \theta\neq 0, \ \text{and} \  \alpha_i\neq 0, \ \ 4\leq i \leq {n-1}.\\
\qquad \quad \ \ \ \ \ 2.\ \ \theta\neq 0, \ \text{and} \  \alpha_i= 0, \ \ 4\leq i \leq {n-1}.\\
\qquad \quad \ \ \ \ \ 3.\ \ \theta= 0,\ \alpha_4 \neq 0,\ \alpha_5 \neq 0,  \ \text{and}\  \alpha_i= 0,  \ \ 6\leq i \leq {n-1}.\\
\qquad \ \ \ \ \ \ \ \ 4.\  \ \theta\neq 0,\ \alpha_4=0  \ \text{and there exists} \ \ \ell\in\{5, 6, 7, 8, .., n-1\}\\ \hfill \text{such that}\  \alpha_i\neq0,\ \text{for all}\ i\geq \ell \ \text{and} \ \  \alpha_i= 0,\ \text{if}\ i < \ell.\\
  n-1,\ \ \text{if}\ \ 1.\ \ \theta\neq 0,\ \alpha_4 \neq 0,\ \alpha_5 \neq 0,  \ \text{and}\ \  \alpha_i= 0,  \ \ 6\leq i \leq {n-1}.\\
\qquad \quad \ \ \ \ \ 2.\ \ \theta\neq 0,\ \alpha_4 = 0,\ \alpha_5 \neq 0,  \ \text{and}\ \ \alpha_i= 0,  \ \ 6\leq i \leq {n-1}.\\
\end{array}
\right.$$
 We shall treat only one case, where $ \theta\neq 0, \alpha_i\neq 0,$ for $ 4\leq i \leq n-1.$ The others cases are similar. Put $ d(e_j)=\sum\limits_{i=j}^nd_i^je_i,$ where $ j=1, 2.$
Then owing to  $ [e_1, e_1]=e_3$ one has
$$d(e_3)=[d(e_1), e_1]+[e_1, d(e_1)]=\left[\sum_{i=1}^nd_i^1e_i, e_1\right]+\left[e_1, \sum_{i=1}^nd_i^1e_i\right]$$
$$=d_1^1[e_1, e_1]+d_2^1[e_2, e_1]+\left[\sum_{i=3}^nd_i^1e_i, e_1\right]+d_1^1[e_1, e_1]+d_2^1[e_1, e_2]$$
$$ =(2d_1^1+d_2^1)e_3+\sum_{i=3}^{n-1}d_i^1e_{i+1}+d_2^1\sum_{i=3}^{n-1}(\alpha_i e_i+\theta e_n)$$
$$ =(2d_1^1+d_2^1)e_3+\sum_{i=4}^{n-1}(d_{i-1}^1+d_2^1\alpha_i)e_i+(d_{n-1}^1+d_2^1\theta )e_n.$$
Due to $ [e_2, e_1]=e_3$ we have
$$d(e_3)=[d(e_2), e_1]+[e_2, d(e_1)]=\left[\sum_{i=1}^nd_i^2e_i, e_1\right]+\left[e_2, \sum_{i=1}^nd_i^1e_i\right]$$
$$=d_2^2[e_2, e_1]+\left[\sum_{i=3}^nd_i^2e_i, e_1\right]+d_1^1[e_2, e_1]+d_2^1[e_2,e_2]=d_2^2e_3+\sum_{i=3}^{n-1}d_i^2e_{i+1}+d_1^1e_3+d_2^1\sum_{i=4}^n\alpha_i e_i$$
$$=(d_1^1+d_2^2)e_3+\sum_{i=4}^{n}d_{i-1}^2e_i+d_2^1\sum_{i=4}^n\alpha_i e_i=(d_1^1+d_2^2)e_3+\sum_{i=4}^{n-1}(d_{i-1}^2+d_2^1\alpha_i )e_i+(d_{n-1}^2+d_2^1\alpha_n )e_n.$$
Comparing the last two expressions for $d(e_3)$  we obtain
\begin{equation}\label{11} d_2^2=d_1^1+d_2^1,  d_i^2=d_i^1 , \  \text{for}\ \ 3\leq i \leq n-1 \ \text{and} \ \  d_{n-1}^2=d_{n-1}^1+d_2^1(\theta-\alpha_n).  \end{equation}
From $ [e_3, e_1]=e_4$ one has
$$d(e_4)=[d(e_3), e_1]+[e_3, d(e_1)]$$
$$=\left[(d_1^1+d_2^2)e_3+\sum_{i=4}^{n-1}(d_{i-1}^2+d_2^1\alpha_i)e_i
+(d_{n-1}^2+d_2^1\alpha_n )e_n, e_1\right]+\left[e_3, \sum_{i=1}^nd_i^1e_i\right]$$
$$=(d_1^1+d_2^2)e_4+\sum_{i=4}^{n-1}(d_{i-1}^2+d_2^1\alpha_i)e_{i+1}
+d_1^1[e_3, e_1]+d_2^1[e_3, e_2]$$
$$=(d_1^1+d_2^2)e_4+\sum_{i=4}^{n-1}(d_{i-1}^2+d_2^1\alpha_i)e_{i+1}
+d_1^1e_4+d_2^1\sum_{i=4}^{n-1}\alpha_ie_{i+1}$$
$$=(2d_1^1+d_2^2)e_4+\sum_{i=5}^{n}(d_{i-2}^2+2d_2^1\alpha_{i-1})e_i.$$
Let consider $ [e_4, e_1]=e_5.$ Then,
$$d(e_5)=[d(e_4), e_1]+[e_4, d(e_1)]$$
$$=\left[(2d_1^1+d_2^2)e_4+\sum_{i=5}^{n}(d_{i-2}^2+2d_2^1\alpha_{i-1})e_i, e_1\right]+\left[e_4, \sum_{i=1}^nd_i^1e_i\right]$$
$$=\left[(2d_1^1+d_2^2)e_4+\sum_{i=5}^{n}(d_{i-2}^2+2d_2^1\alpha_{i-1})e_i, e_1\right]+
d_1^1[e_4, e_1]+d_2^1[e_4, e_2]$$
$$=(2d_1^1+d_2^2)e_5+\sum_{i=5}^{n-1}(d_{i-2}^2+2d_2^1\alpha_{i-1})e_{i+1}+d_1^1e_5
+d_2^1\sum_{i=4}^{n-2}\alpha_ie_{i+2}$$
$$=(3d_1^1+d_2^2)e_5+\sum_{i=5}^{n-1}(d_{i-2}^2+2d_2^1\alpha_{i-1})e_{i+1}
+d_2^1\sum_{i=5}^{n-1}\alpha_{i-1}e_{i+1}$$
$$=(3d_1^1+d_2^2)e_5+\sum_{i=5}^{n-1}(d_{i-2}^2+3d_2^1\alpha_{i-1})e_{i+1}
$$
$$=(3d_1^1+d_2^2)e_5+\sum_{i=6}^{n}(d_{i-3}^2+3d_2^1\alpha_{i-2})e_i.
$$
Similarly,
\begin{equation}\label{12}d(e_k)=((k-2)d_1^1+d_2^2)e_k+\sum_{i=k+1}^{n}(d_{i-k+2}^2+(k-2)d_2^1\alpha_{i-k+3})e_i . \end{equation}
From, $ e_n=[e_{n-1}, e_1]$ we get $$ d(e_n)=[d(e_{n-1}), e_1]+[e_{n-1}, d(e_1)]. $$
The substitution $k$ by $ n-1 $ in (\ref{12}) gives $ d(e_{n-1})={((n-3)d_1^1+d_2^2)e_{n-1}+(d_3^2+(n-3)d_2^1\alpha_4)e_n}$\\
and then
$$ d(e_n)=\left[((n-3)d_1^1+d_2^2)e_{n-1}+(d_3^2+(n-3)d_2^1\alpha_4)e_n, e_1\right]+d_1^1[e_{n-1}, e_1] $$
$$ =\left((n-3)d_1^1+d_2^2\right)e_n+d_1^1e_n .$$
As a result one has
\begin{equation} \label{13} d(e_n)=\left((n-2)d_1^1+d_2^2\right)e_n. \end{equation}
On the other hand,  $$ [e_{n-2}, e_2]=\alpha_4e_n.$$ Notice that
 $ \alpha_4 \neq 0, $ therefore $$d(e_n)=\frac{1}{\alpha_4}d([e_{n-2}, e_2]). $$
This implies that $$d(e_n)=\frac{1}{\alpha_4}([d(e_{n-2}), e_2]+[e_{n-2}, d(e_2)]  ).$$
We substitute $k$ by $ n-2 $ in (\ref{12}), to obtain \\
 $$d(e_{n-2})=((n-4)d_1^1+d_2^2)e_{n-2}+
\sum\limits_{i=n-1}^{n}(d_{i-n+4}^2+(n-4)d_2^1\alpha_{i-n+5})e_i.$$
 Then

 $d(e_n)=\frac{1}{\alpha_4}\left(\left[((n-4)d_1^1+d_2^2)e_{n-2}+
\sum\limits_{i=n-1}^{n}(d_{i-n+4}^2\right.\right.$

 \qquad \qquad \qquad \qquad \qquad $\left.\left.+(n-4)d_2^1\alpha_{i-n+5})e_i, e_2\right]+[e_{n-2}, d(e_2)]  \right)$

\qquad \qquad $=\frac{1}{\alpha_4}\left(\left[((n-4)d_1^1+d_2^2)e_{n-2}+
\sum\limits_{i=n-1}^{n}(d_{i-n+4}^2\right.\right.$

\qquad \qquad \qquad \qquad $\left.\left.+(n-4)d_2^1\alpha_{i-n+5})e_i, e_2\right]
+d_2^2[e_{n-2}, e_2]  \right)$

$$=\frac{1}{\alpha_4}\left(((n-4)d_1^1+d_2^2)\alpha_4e_n
+d_1^2e_{n-1}+d_2^2\alpha_4e_n  \right)$$
$$=\frac{1}{\alpha_4}\left(((n-4)d_1^1+d_2^2)\alpha_4e_n+
+d_1^2e_{n-1}+d_2^2\alpha_4e_n  \right)$$
$$=\frac{d_1^2}{\alpha_4}e_{n-1}+\left((n-4)d_1^1+2d_2^2\right)e_n.$$
Thus\\
\begin{equation}\label{14}d(e_n)=\frac{d_1^2}{\alpha_4}e_{n-1}+\left((n-4)d_1^1+2d_2^2\right)e_n.\end{equation}
Comparing  (\ref{13}) and (\ref{14})  we obtain
\begin{equation}\label{15} d_1^2=0 , d_2^2= 2d_1^1.\end{equation}

The matrix of $ d $  has the form $ D=(d_k^l)_{k,l=1, 2, 3 ,...n},  $ where
$$d^1_2=d^1_1,\
 d^2_2= 2d_1^1,\
  d_1^2=0,\
 d^2_i=d^1_i , \ \ \  3\leq i\leq n-2,\
d_{n-1}^2=d_{n-1}^1+d_2^1(\theta+\alpha_n).  $$
Hence, in this case the dimension of $ Der(L)$ for $L\in FLb_n$ is $n.$
\end{proof}

Now we describe the derivation algebra of elements from $ SLb_n.$
\begin{theorem} The dimensions of the derivation algebras for elements of $SLb_n$ vary between $n-1$ and $n+2.$
\end{theorem}

\begin{proof} Similarly to the case of $FLb_n$ for the class $SLb_n$ we have the distribution for dimension of derivation algebra as follows.

$$dim Der(\emph{L}) =\left\{\begin{array}{lll}
 n+2,\  \ \text{if}\ \ 1.\ \  \gamma = 0 \ \text{and} \  \beta_i=0, \ \ 3\leq i \leq {n-1}.\\
 \qquad \quad \ \ \ \ \ 2.\ \  \gamma = 0 \ \text{and there exists}\ i \in\{3, 4, 5, ..., n-1\}\\
\hfill \text{such that}\  \beta_i\neq 0,\ \beta_j=0 \ \text{for}\ j \neq i. \\
  n+1,\  \ \text{if}\ \  1.\ \   \gamma\neq 0 \ \text{and} \  \beta_i=0, \ \ 3\leq i \leq {n-1}.\\
\qquad \quad \ \ \ \ \  2.\ \   \gamma =0 \ \text{and} \  \beta_i\neq 0, \ \ 3\leq i \leq {n-1}.\\
\qquad \quad \ \ \ \ \  3.\  \ \gamma\neq 0  \ \text{and thee exists} \ \ \ell\in\{3, 4, 5, ..., n-1\}\  \text{such that}\  \beta_i\neq0,\ \text{for all}\ \\ \hfill i\geq \ell \ \text{and} \ \  \beta_i= 0,\ \text{if}\  i < \ell, \ \text{where} \ n=2\ell-1.\\
\qquad \quad \ \ \ \ \  4.\ \   \gamma \neq 0 \ \text{and} \  \beta_i\neq 0, \ \ 3\leq i \leq {n-1}.\\
\qquad \quad \ \ \ \ \ 5.\ \  \gamma \neq 0 \ \text{and there exists}\ i \in\{3, 4, 5, ..., n-1\}\\
\hfill \text{such that}\  \beta_i\neq 0,\ \beta_j=0 \ \text{for}\ j \neq i. \\
\qquad \quad \ \ \ \ \ 6.\ \  \gamma \neq 0 \ \text{and there exists}\ i \in\{3, 4, 5, ..., n-1\}\\
\hfill \text{such that}\  \beta_i= 0,\ \beta_j\neq 0 \ \text{for}\ j \neq i. \\
\qquad \quad \ \ \ \ \ 7.\ \  \gamma = 0 \ \text{and there exists}\ i \in\{3, 4, 5, ..., n-1\}\\
\hfill \text{such that}\  \beta_i= 0,\ \beta_j\neq 0 \ \text{for}\ j \neq i. \\
n, \  \quad \ \  \text{if}\ \ \ 1.\  \ \gamma\neq 0  \ \text{and there exists}\ \ \ \ell\in\{3, 4, 5, ..., n-1\}\ \text{such that}\  \beta_i\neq0,\ \text{for all}\\ \hfill \ i\geq \ell \ \text{and} \ \  \beta_i= 0,\ \text{if}\ i < \ell, \ \text{where} \ n\neq 2\ell-1.\\
\qquad \quad \ \ \ \ \ 2.\ \   \gamma \neq 0,\ \beta_{n-1}\neq 0 ,  \ \text{and} \  \beta_i= 0, \ \ 3\leq i \leq {n-2}.\\
\qquad \quad \ \ \ \ \ 3.\ \    \gamma \neq 0,\ \beta_{n-1}= 0,\ \beta_{n-2}\neq 0 \ \text{and} \  \beta_i= 0, \ \ 3\leq i \leq {n-3}.\\
\qquad \quad \ \ \ \ \ 4.\ \   \gamma = 0,\ \beta_{n-1}= 0,\ \beta_{n-2}\neq 0 \ \text{and} \  \beta_i= 0, \ \ 3\leq i \leq{n-3}.\\
\qquad \quad \ \ \ \ \  5.\ \  \gamma\neq 0,\  \beta_{n-1}\neq 0,\ \beta_{n-2}\neq 0   \ \text{and}\  \beta_i= 0,  \ \ 3\leq i \leq {n-3}.\\
 n-1, \ \ \text{if} \ \  1.\ \   \  \gamma= 0,\  \beta_{n-1}\neq 0,\ \beta_{n-2}\neq 0   \ \text{and}\  \beta_i= 0,  \ \ 3\leq i \leq {n-3}.\\
\qquad \quad \ \ \ \ \ 2.\ \    \gamma = 0,\ \beta_{3}\neq0 ,  \ \text{and} \  \beta_i= 0, \ \ 4\leq i \leq {n-1}.\\
\end{array}
\right.$$

Let $ \gamma= 0,\ \beta_{n-1}\neq 0,\ \beta_{n-2}\neq 0 $ and $ \beta_i= 0, $ for $i=3,4 ,5 , .., n-3.$
Put $ d(e_j)=\sum\limits_{i=j}^nd_i^je_i,$ where $ j=1,2. $
Since, $ [e_1, e_1]=e_3$, we have
$$d(e_3)=[d(e_1), e_1]+[e_1, d(e_1)]=\left[\sum_{i=1}^nd_i^1e_i, e_1\right]+\left[e_1, \sum_{i=1}^nd_i^1e_i\right]$$
$$=d_1^1[e_1, e_1]+d_2^1[e_2, e_1]+\left[\sum_{i=3}^nd_i^1e_i, e_1\right]+d_1^1[e_1, e_1]+d_2^1[e_1, e_2]$$
$$=d_1^1e_3+\sum_{i=3}^{n-1}d_i^1e_{i+1}+d_1^1e_3+d_2^1\beta_{n-2}e_{n-1}+d_2^1\beta_{n-1}e_{n}$$
$$=(2d_1^1)e_3+\sum_{i=3}^{n-1}d_i^1e_{i+1}+d_2^1\beta_{n-2}e_{n-1}+d_2^1\beta_{n-1}e_{n}$$
$$=(2d_1^1)e_3+\sum_{i=4}^{n-2}d_{i-1}^1e_i+(d_{n-2}^1+d_2^1\beta_{n-2})e_{n-1}+(d_{n-1}^1+d_2^1\beta_{n-1})e_{n}.$$
Thus,\\
\begin{equation}d(e_3)=(2d_1^1)e_3+\sum_{i=4}^{n-2}d_{i-1}^1e_i+(d_{n-2}^1+d_2^1\beta_{n-2})e_{n-1}
+(d_{n}^1+d_2^1\beta_{n-1})e_{n}.\end{equation}
From $ [e_2, e_1]=0 $ we get $$0=[d(e_2), e_1]+[e_2, d (e_1)] =\left[\sum_{i=2}^nd_i^2e_i, e_1\right]+\left[e_2, \sum_{i=1}^nd_i^1e_i\right] $$
$$=d_2^2[e_2, e_1]+\left[\sum_{i=3}^nd_i^2e_i, e_1\right]+d_1^1[e_2, e_1]+d_2^1[e_2, e_2] $$
$$=\left[\sum_{i=3}^nd_i^2e_i, e_1\right]=\sum_{i=4}^{n-1}d_{i-1}^2e_{i}.$$
Therefore we obtain
\begin{equation}\begin{array}{lll}
d^2_i=0 , \ \ \  3\leq i\leq n-1.
\end{array}
\end{equation}
Consider $[e_3, e_1]=e_4$. Then
 $$d(e_4)=[d(e_3), e_1]+[e_3, d(e_1)]   $$
$$=\left[(2d_1^1)e_3+\sum_{i=4}^{n-2}d_{i-1}^1e_i+d_2^1(d_{n-1}^1+\beta_{n-2})e_{n-1}
+(d_{n}^1+d_2^1\beta_{n-1})e_{n}, e_1\right]+\left[e_3, \sum_{i=1}^nd_i^1e_i\right]   $$
$$=(2d_1^1)e_4+\sum_{i=4}^{n-2}d_{i-1}^1e_{i+1}+(d_{n-1}^1+d_2^1\beta_{n-2})e_{n}
+\left[e_3, \sum_{i=1}^nd_i^1e_i\right]   $$
$$=(2d_1^1)e_4+\sum_{i=4}^{n-2}d_{i-1}^1e_{i+1}+(d_{n-1}^1+d_2^1\beta_{n-2})e_{n}
+d_1^1e_4   $$
$$=(3d_1^1)e_4+\sum_{i=4}^{n-2}d_{i-1}^1e_{i+1}+(d_{n-1}^1+d_2^1\beta_{n-2})e_{n}.
$$
Thus,\\
\begin{equation}d(e_4)=(3d_1^1)e_4+\sum_{i=4}^{n-2}d_{i-1}^1e_{i+1}+(d_{n-1}^1+d_2^1\beta_{n-2})e_{n}.
\end{equation}
Take $[e_4, e_1]=e_5$. Then
 $$d(e_5)=[d(e_4), e_1]+[e_4, d(e_1)]   $$
$$=\left[(3d_1^1)e_4+\sum_{i=4}^{n-2}d_{i-1}^1e_{i+1}+d_2^1(d_{n-1}^1+\beta_{n-2})e_{n}, e_1\right]+\left[e_4, \sum_{i=1}^nd_i^1e_i\right]   $$
$$=(3d_1^1)e_5+\sum_{i=4}^{n-2}d_{i-1}^1e_{i+2}+d_i^1e_5   =(4d_1^1)e_5+\sum_{i=6}^{n}d_{i-3}^1e_{i}. $$
Hence,
\begin{equation}d(e_5)=(4d_1^1)e_5+\sum_{i=6}^{n}d_{i-3}^1e_{i}.
\end{equation}
Similarly,\begin{equation} \label{103} d(e_k)=(k-1)d_1^1e_k+\sum_{i=k+1}^{n}d_{i-k+2}^1e_{i} .  \end{equation}
It is clear that this relation is true for $k\geq 5$.
Consider, $ e_n=[e_{n-1}, e_1]$ then $  d(e_n)=d([e_{n-1}, e_1]).$
So, $$ d(e_n)=[d(e_{n-1}), e_1]+\left[e_{n-1}, d(e_1)\right]. $$
We substitute $k$ by $ n-1 $ in (\ref{103}), and obtain
$$ d(e_n)=[(n-2)d_1^1e_{n-1}+d_{3}^1e_{n} , e_1]+
\left[e_{n-1}, \sum_{i=1}^nd_i^1e_i\right]
=((n-2)d_1^1)e_{n}+d_1^1e_n
 =((n-1)d_1^1)e_{n}.$$
Thus,
\begin{equation} \label{56} d(e_n)=((n-1)d_1^1)e_{n} .
\end{equation}
On the other hand,  $$ [e_1, e_2]=\beta_{n-2}e_{n-1}+\beta_{n-1}e_n.$$
Then $$d(e_n)=\frac{1}{\beta_{n-1}}\left(d([e_1, e_2])-\beta_{n-2}d(e_{n-1})\right).$$
This implies, $$=\frac{1}{\beta_{n-1}}\left([d(e_1), e_2]+[e_1, d(e_2)]-\beta_{n-2}d(e_{n-1})\right)$$
$$=\frac{1}{\beta_{n-1}}\left([\sum_{i=1}^nd_i^1e_i, e_2]+[e_1, \sum_{i=2}^nd_i^2e_i]-\beta_{n-2}d(e_{n-1})\right)$$
$$=\frac{1}{\beta_{n-1}}\left(d_1^1[e_1, e_2]+d_2^1[e_2, e_2]+d_3^1[e_3, e_2]+d_2^2[e_1, e_2]-\beta_{n-2}d(e_{n-1})\right)$$
$$=\frac{1}{\beta_{n-1}}\left(d_1^1(\beta_{n-2}e_{n-1}+\beta_{n-1}e_{n})+d_3^1\beta_{n-2}e_n
+d_2^2(\beta_{n-2}e_{n-1}+\beta_{n-1}e_{n})-\beta_{n-2}d(e_{n-1})\right)$$
$$=\frac{1}{\beta_{n-1}}\left((d_1^1+d_2^2)(\beta_{n-2}e_{n-1}+\beta_{n-1}e_{n})
+d_3^1\beta_{n-2}e_n-\beta_{n-2}( (n-2)d_1^1e_{n-1}+d_3^1e_{n})\right)$$
$$=\frac{1}{\beta_{n-1}}\left(((3-n)d_1^1+d_2^2)\beta_{n-2}e_{n-1}+((d_1^1
+d_2^2)\beta_{n-1}+d_3^1\beta_{n-2}e_n-\beta_{n-2} d_3^1)e_{n}\right).$$
Thus,\\
\begin{equation} \label{57} d(e_n)=\frac{1}{\beta_{n-1}}\left(((3-n)d_1^1+d_2^2)\beta_{n-2}e_{n-1}\right)
+((d_1^1+d_2^2))e_n.\end{equation}
Comparing  (\ref{56}) and (\ref{57})  we obtain
\begin{equation} \label{58} \begin{array}{lll}(3-n)d_1^1+d_2^2=0 \ \text{ that is}\ \ d_2^2=(n-3)d_1^1 \\ \text{ and}\
 (d_1^1+d_2^2)=(n-1)d_1^1 \ \text{which implies}\ \ d_2^2=(n-2)d_1^1.\end{array}\end{equation}
 From (\ref{58}) we get
 \begin{equation}
 d_1^1=0,\  \text{and}\  d_2^2=0.
 \end{equation}
 The matrix of $ d $  has the form $ D=(d_k^l)_{k,l}=1, 2, 3 ,...n,  $ where
\begin{equation}\begin{array}{lll}
$$d_2^2=d_1^1=0, \ \  d_2^1=0 ,$$\\
 d_1^2=0 \ \text{and}\
d^2_i=0 , \ \ \  3\leq i\leq n-1.
\end{array}
\end{equation}
The dimension of the derivation algebra of $L\in SLb_n$ is $ n-1.$

The other cases are treated similarly.

%For algebras $L$ from $SLb_n$ with the constraints  $ \gamma\neq 0, \beta_{n-1}\neq 0 ,\beta_{n-2}\neq 0, $ $ \beta_i= 0, $  $i=3,4 ,5 , .., n-3,$ for the structure constants of $L$ we have $Dim Der(L)=n.$
%
%Let now $L \in SLb_n$ with $\gamma\neq 0 $ and $ \alpha_i\neq 0, $ for $i=4,5,..,n-1$. Then the dimension of $Der(L)$ for the $L$ is  $n+1.$

And at last, if $L \in SLb_n$ with $ \gamma=0 $ and $ \alpha_i=0, $ for $i=4,5,..,n-1$. Then dimension of the derivation algebra of $L$ is $n+2,$ which is immediate from the case of $NGF_2.$

\end{proof}

\section{Acknowledgements}

\indent\indent  The work was supported by FRGS Project 01-12-10-978FR MOHE (Malaysia).

\end{document}